\documentclass[a4paper,notitlepage,twoside,leqno,10pt]{amsart}

\usepackage{bbm,pifont,latexsym}

\usepackage{dcolumn,indentfirst}
\usepackage[bookmarksopen=true,pdfstartview=FitH]{hyperref}
\usepackage{amsmath,amssymb,amscd,amsthm,amsfonts,mathrsfs}
\usepackage{color,graphicx,xcolor,graphics,subfigure,extarrows,caption2}
\usepackage{titletoc}

\newtheorem{thm}{Theorem}[section]
\newtheorem{lema}[thm]{Lemma}
\newtheorem{cor}[thm]{Corollary}

\theoremstyle{definition}
\newtheorem*{defi}{Definition}
\newtheorem*{rmk}{Remark}%联合编号

\newcommand{\D}{\mathbb{D}}

\newcommand{\R}{\mathbb{R}}
\newcommand{\Z}{\mathbb{Z}}
\newcommand{\N}{\mathbb{N}}

\newcommand{\C}{\mathbb{C}}

\newcommand{\ME}{\mathcal{E}}
\newcommand{\MQ}{\mathcal{Q}}

\newcommand{\diam}{\textup{diam}}

\newcommand{\ii}{\textup{i}}

\newcommand{\re}{\textup{Re\,}}
\newcommand{\im}{\textup{Im\,}}
\newcommand{\sing}{\textup{sing}}
\newcommand{\dens}{\textup{density}}
\newcommand{\Area}{\textup{Area}}

\begin{document}

\author{SONG ZHANG}
\address{Department of Mathematics, Nanjing University, Nanjing, 210093, P. R. China}
\email{dg1521017@smail.nju.edu.cn}

\author{FEI YANG}
\address{Department of Mathematics, Nanjing University, Nanjing, 210093, P. R. China}
\email{yangfei@nju.edu.cn}

%---------------------------------------------------------------------------------------------------------------
\title[Area of the complement of the fast escaping sets]{AREA OF THE COMPLEMENT OF THE FAST ESCAPING SETS OF A FAMILY OF ENTIRE FUNCTIONS}

\begin{abstract}
Let $f$ be an entire function with the form $f(z)=P(e^z)/e^z$, where $P$ is a polynomial with $\deg(P)\geq2$ and $P(0)\neq 0$. We prove that the area of the complement of the fast escaping set (hence the Fatou set) of $f$ in a horizontal strip of width $2\pi$ is finite. In particular, the corresponding result can be applied to the sine family $\alpha\sin(z+\beta)$, where $\alpha\neq 0$ and $\beta\in\C$.
\end{abstract}

% AMS subject classifications (used in AMS journals)

% AMS keywords (used in AMS journals)

% today's date, or fill in whatever date you prefer
\date{\today}

% acknowledge support, etc
% \thanks{This research was partially supported by NSF grant DOA-123456789.}
% \thanks{We would like to thank our colleagues for their helpful criticism.}

% dedication
% \dedicatory{Dedicated to Professor Donald Knuth on the occasion of his $100$th birthday}

\maketitle

%----------------------------------------------------------------------------------------------------------------
%\vskip1.0cm
%\tableofcontents
%----------------------------------------------------------------------------------------------------------------

%----------------------------------------------------------------------------------------------------------------
\section{Introduction}

Let $f:\C\to\C$ be a transcendental entire function. Denote by $f^{\circ n}$ the $n$-th iterate of $f$. The \textit{Fatou set} $F(f)$ of $f$ is defined as the maximal open set in which the family of iterates $\{f^{\circ n}:n\in\N\}$ is normal in the sense of Montel. The complement of $F(f)$ is called the \textit{Julia set} of $f$, which is denoted by $J(f)$. It is well known that $J(f)$ is a perfect completely invariant set which is either nowhere dense or coincides with $\mathbb{C}$. For more details about these sets, one can refer \cite{Bea91}, \cite{CG93} and \cite{Mil06} for rational maps, and \cite{Ber93} and \cite{EL92} for meromorphic functions.

Already in 1920s, Fatou considered the iteration of transcendental entire functions \cite{Fat26} and one of his study object was $f(z)=\alpha\sin(z)+\beta$, where $0<\alpha<1$ and $\beta\in\R$. After Misiurewicz showed that the Fatou set of $f(z)=e^z$ is empty in 1981 \cite{Mis81}, the dynamics of exponential maps and trigonometric functions attracted many interests from then on. See \cite{DK84}, \cite{DT86} and \cite{DG87} for example. In particular, in 1987 McMullen \cite{McM87} proved a remarkable result which states that the Julia set of $\sin(\alpha z+\beta)$, $\alpha\neq 0$ always has positive Lebesgue area and the Hausdorff dimension of the Julia set of $\lambda e^z$, $\lambda\neq 0$ is always $2$. From then on a series of papers considered the area and the Hausdorff dimension of the dynamical objects of the transcendental entire functions, not only for the Julia sets in dynamical planes (see \cite{Sta91}, \cite{Kar99a}, \cite{Kar99b}, \cite{Tan03}, \cite{Sch07}, \cite{Bar08}, \cite{RS10}, \cite{AB12}, \cite{Rem14}, \cite{Six15a} and the references therein for example), but also the bifurcation loci in the parameter spaces (see \cite{Qiu94} and \cite{ZL12}).

Unlike the polynomials, the Julia set of a transcendental entire function $f$ is always unbounded. Since the Fatou set of $f$ is dense in the complex plane (if $F(f)\neq\emptyset$), it is interesting to ask when the Fatou set of $f$ has finite area. For the sine function $f(z)=\sin z$, Milnor conjectured that the area of the Fatou set of $f$ is finite in a vertical strip of width $2\pi$. By applying the tools in \cite{McM87}, Schubert proved this conjecture in 2008 \cite{Sch08}.

For a transcendental entire function $f$, the \textit{escaping set} $I(f)$ was studied firstly by Eremenko in \cite{Ere89}. A subset of the escaping set, called the \textit{fast escaping set} $A(f)$, was introduced by Bergweiler and Hinkkanen in \cite{BH99}. These sets have received quite a lot of attention recently. Especially for the fast escaping set, see \cite{Six11}, \cite{RS12}, \cite{Six13}, \cite{Six15b}, \cite{Evd16} and the references therein. In this paper, we consider the area of the complement of the fast escaping sets of a family of entire functions and try to extend the result of Schubert to this class. Our main result is the following.

\begin{thm}\label{thm}
Let $P$ be a polynomial with $\deg(P)\geq 2$ and $P(0)\neq 0$. Then the area of the complement of the fast escaping set of any function with the form $f(z)=P(e^{z})/e^{z}$ is finite in any horizontal strip of width $2\pi$.
\end{thm}

The method in this paper is strongly inspired by the work of McMullen and Schubert (\cite{McM87} and \cite{Sch08}). It is worth to mention that we give also a specific formula of the upper bound of $\Area(S\cap A(f)^c)$ in terms of the coefficients of the polynomial $P$ (see Theorem \ref{thm-main-restate}), where $S$ is any horizontal strip of width $2\pi$ and $A(f)^c$ is the complement of the fast escaping set of $f$. In fact, we believe that our method can be adopted also to the type of entire functions with the form
\begin{equation*}
f(z)=\frac{P(w)}{w^m}\circ\exp(z)
\end{equation*}
completely similarly, where $m\geq 1$ is a positive integer, $P$ is a polynomial with degree $\deg(P)\geq m+1$ and $P(0)\neq 0$.

As a consequence of Theorem \ref{thm} and Theorem \ref{thm-main-restate}, we have the following result on the area of the complement of the fast escaping set of the sine family.

\begin{thm}\label{thm-sin}
Let $S$ be any vertical strip of width $2\pi$. Then the area of the complement of the fast escaping set of $f(z)=\alpha \sin(z+\beta)$ with $\alpha\neq 0$ satisfies
\begin{equation*}
\Area(S\cap A(f)^c)\leq (4\pi+4r)\left(x^{*}+r+8c\,e^{4-x^*/2}\frac{r}{1-e^{-r/2}}\right),
\end{equation*}
where
\begin{equation*}
r=\frac{1}{8}, \quad c=\frac{536\sqrt{2}}{|\alpha|}+\frac{1}{|\alpha|^2}
\end{equation*}
and
\begin{equation*}
x^*=\max\Big\{\log\Big(1+\frac{18K}{|\alpha|}\Big),\log\Big(\frac{8(K+1)}{|\alpha|}\Big), 6\log 2,12+2\log c\Big\}
\end{equation*}
with $K=\max\{|\alpha|/2,|\beta|\}$. In particular, if $f(z)=\sin z$ or $\cos z$, then
\begin{equation*}
\Area(S\cap A(f)^c)<361.
\end{equation*}
\end{thm}

Since the fast escaping set of $f(z)=P(e^{z})/e^{z}$ is contained in the Julia set (see Corollary \ref{cor-I-J}), it means that the complement of the fast escaping set contains the Fatou set and hence Theorem \ref{thm-sin} is a generalization of Schubert's result. In \cite{Sch08} Schubert proved that $\Area(S\cap F(f))<574$ for $f(z)=\sin z$, where $S$ is a vertical strip with width $2\pi$. See Figure \ref{Fig-sin-cos}.

\begin{figure}[!htpb]
  \setlength{\unitlength}{1mm}
  \centering
  \includegraphics[height=45mm]{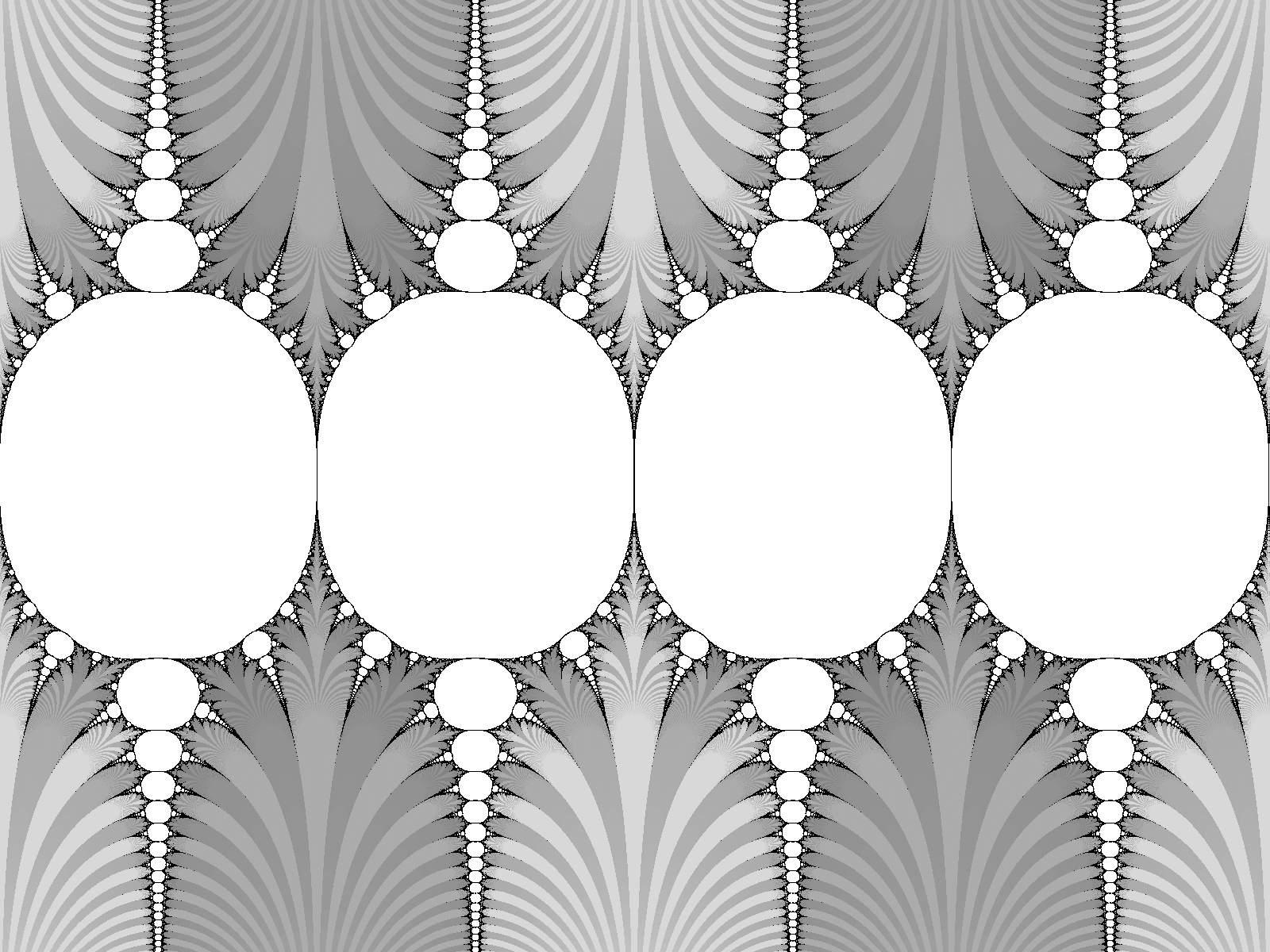}\hskip0.3cm
  \includegraphics[height=45mm]{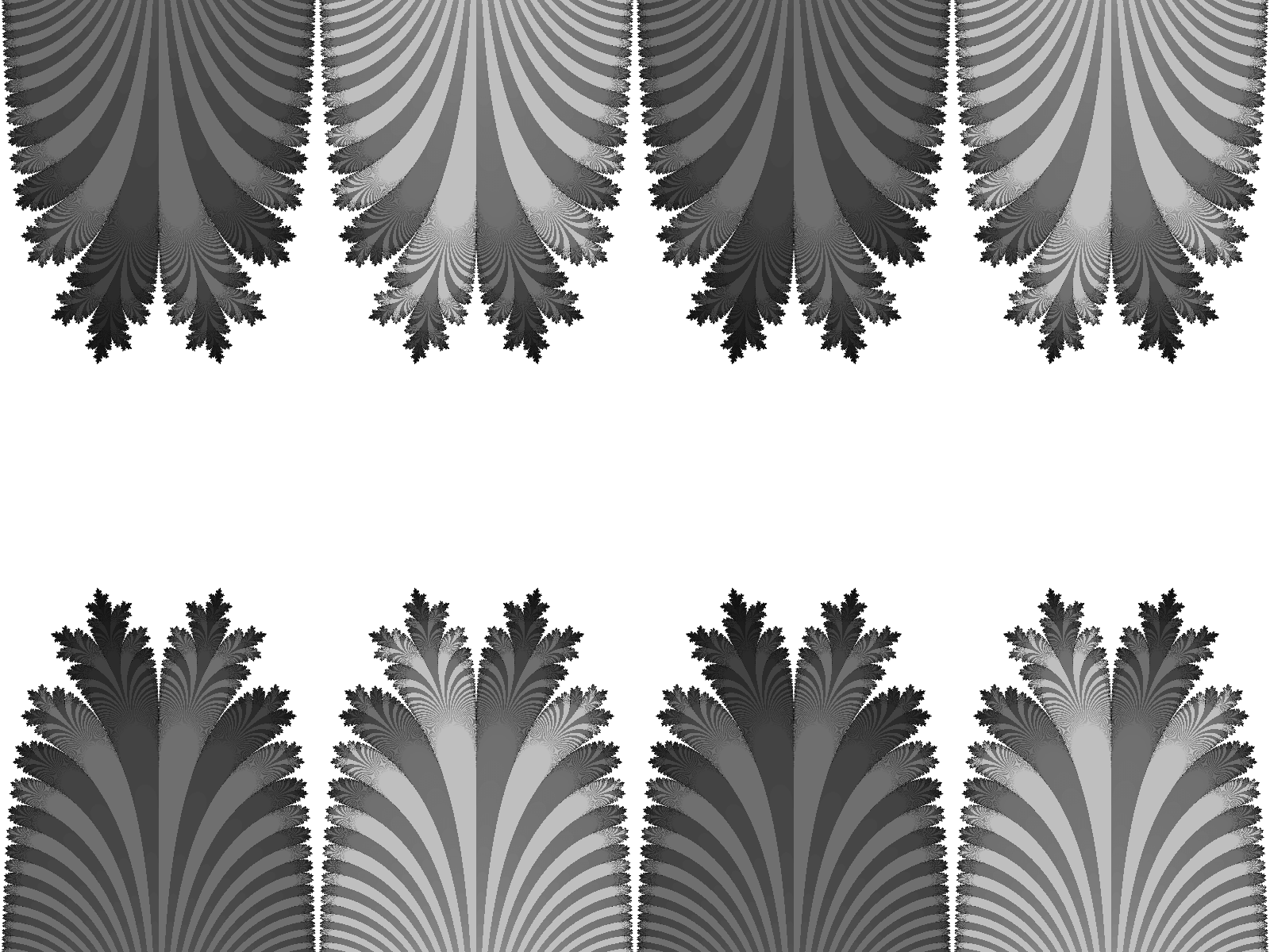}
  \caption{The Fatou sets (white regions) of $f(z)=\sin z$ and $f(z)=\cos z$. Both of these functions have period $2\pi$. It is shown in Theorem \ref{thm-sin} that the area of the complement of the fast escaping set (hence the Fatou set) of $f$ in a vertical strip with width of $2\pi$ is bounded above by $361$.}
  \label{Fig-sin-cos}
\end{figure}

\vskip0.2cm
We collect some notations which will be used throughout of this paper. Let $\N$, $\Z$, $\R$ and $\C$, respectively, be the set of natural numbers, integers, real numbers and complex numbers. For any $x\geq 0$, we use $[x]$ to denote the integer part of $x$. Hence $x-1< [x]\leq x$. For a subset $X$ of $\C$, we use $X^c$ to denote the complement of $X$ in $\C$. All the distance and diameter in this paper are measured in the Euclidean metric and the area is regarded as the two-dimensional planar Lebesgue area. We use $\D(a,r):=\{z\in\C:|z-a|<r\}$ to denote the round disk with center $a\in\C$ and radius $r>0$.

\vskip0.2cm
\noindent\textit{Acknowledgements.} This work is supported by the National Natural Science Foundation of China (grant Nos.\,11671092, 11671191) and the Fundamental Research Funds for the Central Universities (grant No.\,0203-14380013). We would like to thank Lasse Rempe-Gillen for valuable comments which improved the statements of the main results in this paper and Liangwen Liao for helpful conversations.

\section{Distortion lemmas and some basic settings}

\subsection{Distortion quantities}

As in \cite{McM87} and \cite{Sch08}, we introduce some quantities of distortion in this subsection. Let $D$ be a bounded set in the complex plane $\C$ and let $f$ be a holomorphic function defined in a neighbourhood of $D$. We say that $f$ has \textit{bounded distortion} on $D$ if there are positive constants $c$ and $C$, such that for all distinct $x$ and $y$ in $D$, one has
\begin{equation}\label{1}
c<\frac{|f(x)-f(y)|}{|x-y|}<C.
\end{equation}
The quantity
\begin{equation*}
L(f|_D):=\inf{\{{C}/{c}: c \text{ and } C\text{ satisfy } \eqref{1}\}}
\end{equation*}
is the \textit{distortion} of $f$ on $D$. By \eqref{1} we have
\begin{equation*}
\sup_{z\in D}{|f'(z)|}\leq C  \text{\quad and\quad } \inf_{z\in D}{|f'(z)|}\geq c.
\end{equation*}
Therefore, $L(f|_D)$ has a lower bound satisfying
\begin{equation}\label{Lf}
L(f|_D)\geq \frac{\sup_{z\in D}{|f'(z)|}}{\inf_{z\in D}{|f'(z)|}}.
\end{equation}
The equality holds in this inequality if $D$ is a convex domain.

Let $\Area(E)$ be the Lebesgue area of the measurable set $E\subset\mathbb{C}$. If $X$ and $D$ are two measurable subsets of the complex plane with $\Area(D)>0$, we use
\begin{equation*}
\text{density}(X,D):=\frac{\Area(X\cap D)}{\Area(D)}
\end{equation*}
to denote the \textit{density} of $X$ in $D$. If $c$ and $C$ satisfy \eqref{1}, then $c^{2}\Area(X)\leq \Area(f(X))\leq C^{2}\Area(X)$. This means that
\begin{equation}\label{density}
\dens(f(X),f(D))\leq L(f|_D)^{2}\,\dens(X,D).
\end{equation}
The \emph{nonlinearity} of $f$ on $D$ is defined as
\begin{equation}\label{equ-nonlinearity}
N(f|_D):=\sup{\left\{\frac{|f''(z)|}{|f'(z)|}: z\in D\right\}}\cdot\text{diam}(D),
\end{equation}
provided the right-hand side is finite. In the following by \textit{square} we mean a closed square whose sides are parallel to the coordinate axes. We will use the following relation between the distortion and nonlinearity on squares.

\begin{lema}\label{LN}
Let $Q$ be a compact and convex domain in $\C$ (in particular if $Q$ is a square) and let $f$ be a conformal map defined in a neighbourhood of $Q$ with $N(f|_Q)<1$. Then
$$L(f|_Q)\leq 1+2 N(f|_Q).$$
\end{lema}

\begin{proof}
Since $f$ is conformal, let $z_0$ be a point in $Q$ such that
\begin{equation*}
|f'(z_0)|=\sup_{z\in Q}|f'(z)|>0.
\end{equation*}
Since $Q$ is convex, for any $z\in Q$ we have
\begin{equation*}
\begin{split}
\frac{|f'(z)-f'(z_0)|}{|f'(z_0)|}
=&~\frac{|\int_{z_0}^z f''(\zeta)d\zeta|}{|f'(z_0)|}\leq \frac{\sup_{z\in Q}|f''(z)|}{|f'(z_0)|}\cdot |z-z_0|\\
\leq &~\sup_{z\in Q}\left\{\frac{|f''(z)|}{|f'(z)|}\right\}\cdot\text{diam}(Q)=N(f|_Q)<1.
\end{split}
\end{equation*}
Therefore, the image of $Q$ under $f'(z)$ is contained in the disk $\D(f'(z_0),|f'(z_0)|)$ and hence $\log f'(z)$ is well-defined on $Q$.

Since $Q$ is compact, let $z_1\in Q$ such that
\begin{equation*}
|f'(z_1)|=\inf_{z\in Q}|f'(z)|>0.
\end{equation*}
Since $Q$ is convex and $\log f'(z)$ is well-defined, we have
\begin{equation*}
\begin{split}
\log L(f|_Q)
=&~\log\frac{|f'(z_0)|}{|f'(z_1)|}\leq |\log f'(z_1)-\log f'(z_0)|\\
= &~\left|\int_{z_0}^{z_1}(\log f'(z))'dz\right|=\left|\int_{z_0}^{z_1}\frac{f''(z)}{f'(z)}dz\right|\\
\leq &~\sup_{z\in Q}\left\{\frac{|f''(z)|}{|f'(z)|}\right\}\cdot\text{diam}(Q)=N(f|_Q).
\end{split}
\end{equation*}
Since $e^x\leq 1+2x$ for $x\in[0,1)$, we have
\begin{equation*}
L(f|_Q)\leq \exp(N(f|_Q))\leq 1+2 N(f|_Q).\qedhere
\end{equation*}
\end{proof}

\begin{rmk}
McMullen notes in \cite{McM87} that $L(f|_Q)$ is bounded above by $1+O(N(f|_Q))$ if $N(f|_Q)$ is small. After that Schubert states in \cite{Sch08} that $L(f|_Q)\leq 1+8 N(f|_Q)$ if $N(f|_Q)<1/4$ but without a proof.
\end{rmk}

Let $n$ be a positive integer. For each $1\leq i\leq n$, let $D_{i}\subset \mathbb{C}$ be an open set and $f_{i}: D_{i} \rightarrow \mathbb{C}$ a conformal map. Let $\sigma$ and $M$ be two positive constants satisfying
\begin{equation*}
|f_{i}'(z)|>\sigma>1 \text{\quad and\quad } \frac{|f''_{i}(z)|}{|f'_{i}(z)|}<M, \text{\quad where } z\in D_{i} \text{ and }1\leq i\leq n.
\end{equation*}
Furthermore, let $Q_{i}\subset D_{i}$, $1\leq i\leq n$ be squares with sides of length $r>0$ satisfying $Q_{i+1}\subset f_{i}(Q_{i})$ for all $1\leq i\leq n-1$. Define $V:=f_{n}(Q_{n})$ and \begin{equation*}
F:=(f_{n}\circ \cdots\circ f_{1})^{-1}: V\rightarrow Q_{1}.
\end{equation*}
Then $F$ is a conformal map. McMullen proved that the distortion of $F$ on $V$ is bounded above by a constant depending only on $\sigma$, $M$ and $r$, but not on $f_i$ and $n$ (\cite{McM87}). Actually, this upper bound can be formulated in the following lemma.

\begin{lema}\label{distortion}
If the sides of length $r$ of $Q_i$ is chosen such that $r\leq 1/(4M)$ for all $1\leq i\leq n$, then the distortion of $F$ on $V$ satisfies
\begin{equation*}
 L(F|_V)\leq \exp{\left(\frac{\sigma}{\sigma-1}\right)}.
\end{equation*}
\end{lema}

\begin{proof}
Let $g_i$ be the inverse of $f_i$ which maps $f_i(Q_i)$ to $Q_i$ for $1\leq i\leq n$. Recall that $V=f_{n}(Q_{n})$. Define $V_i:=g_i\circ\cdots\circ g_{n}(V)$, where $1\leq i\leq n$. In particular, $V_n=g_n(V)=Q_n$. Since $|f_{i}'(z)|>\sigma>1$ for all $1\leq i\leq n$, we have
\begin{equation*}
\diam(V_i)\leq\frac{\sqrt{2}r}{\sigma^{n-i}}, \text{\quad for all } 1\leq i\leq n.
\end{equation*}
Note that $V_i\subset Q_i\subset D_i$ for $1\leq i\leq n$ since $Q_{i+1}\subset f_{i}(Q_{i})$ for all $1\leq i\leq n-1$. This means that there exists a square $Q_i'\subset Q_i$ such that $V_i\subset Q_i'$ and the length of the sides of $Q_i'$ is at most $\sqrt{2}r/\sigma^{n-i}$.
Hence by \eqref{equ-nonlinearity}, the nonlinearity of $f_i$ on $Q_i'$ satisfies
\begin{equation*}
N(f_i|_{Q_i'})=\left(\sup_{z\in Q_i'}\frac{|f_i''(z)|}{|f_i'(z)|}\right)\cdot\text{diam}(Q_i')\leq \frac{2Mr}{\sigma^{n-i}}\leq \frac{1}{2}.
\end{equation*}
By Lemma \ref{LN}, we have
\begin{equation*}
L(f_i|_{Q_i'})\leq 1+\frac{4Mr}{\sigma^{n-i}},  \text{\quad for all } 1\leq i\leq n.
\end{equation*}

For any holomorphic functions $f$ and $g$, it is straightforward to verify that the distortion of $f$ and $g$ satisfies\footnote{We suppose that the inverse of $f$ exists in the first equality.}
\begin{equation*}
L(f|_V)=L(f^{-1}|_{f(V)}) \quad \text{and} \quad L((g\circ f)|_{V})\leq L(f|_V)L(g|_{f(V)}).
\end{equation*}
Hence, we have
\begin{align*}
& L(F|_V)=L((f_{n}\circ\cdots\circ f_{1})|_{V_{1}}) \\
\leq & L(f_{1}|_{V_{1}})L(f_{2}|_{V_{2}})\cdots L(f_{n}|_{V_{n}}) \leq L(f_{1}|_{Q_1'})L(f_{2}|_{Q_2'})\cdots L(f_{n}|_{Q_n'})\\
\leq & \prod\limits_{i=0}^{n-1}\left(1+\frac{4Mr}{\sigma^i}\right)\leq\prod\limits_{i=0}^{n-1}\left(1+\frac{1}{\sigma^i}\right).
\end{align*}
Since $\log(1+x)\leq x$ for all $x>0$, we have
\begin{equation*}
L(F|_V)\leq\exp\left(\sum_{i=0}^{n-1}\frac{1}{\sigma^i}\right)
<\exp\left(\sum_{i=0}^{\infty}\frac{1}{\sigma^i}\right)=\exp{\left(\frac{\sigma}{\sigma-1}\right)}.\qedhere
\end{equation*}
\end{proof}

\subsection{Nesting conditions, density and area}

In his proof of the existence of Julia sets of entire functions having positive area, McMullen introduced a system of compact sets which satisfies the nesting conditions \cite{McM87}. We now recall the precise definition.

\begin{defi}[{Nesting conditions}]
For $k\in\N$, let $\mathcal{E}_k$ be a finite collection of measurable subsets of $\mathbb{C}$, i.e. $\ME_k:=\{E_{k,i}: 1\leq i\leq d_{k}\},$  where each $E_{k,i}$ is a measurable subset of $\mathbb{C}$ and $d_{k}:= \# \ME_k<+\infty$. We say that $\{\ME_k\}^{\infty}_{k=0}$ satisfies the \textit{nesting conditions} if $\ME_0=\{E_{0, 1}\}$, where $E_{0, 1}$ is a compact connected measurable set and for all $k\in \mathbb{N}$,
\begin{enumerate}
\item every $E_{k+1,i}\in \ME_{k+1}$ is contained in a $E_{k,j}\in \ME_k$, where $1\leq i\leq d_{k+1}$ and $1\leq j\leq d_k$;
\item every $E_{k,i}\in \ME_k$ contains a $E_{k+1,j}\in \ME_{k+1}$, where $1\leq i\leq d_k$ and $1\leq j\leq d_{k+1}$;
\item $\Area(E_{k,i}\cap E_{k,j})=0$ for all $1\leq i,j\leq d_{k}$ with $i\neq j$; and
\item there is $\rho_{k}>0$ such that for all $1\leq i\leq d_k$ and $E_{k,i}\in \ME_k,$ we have\footnote{Note that $\ME_k$ is a collection of measurable sets for $k\in\N$. For simplicity, sometimes we will not distinguish $\ME_k$ and the union of its elements $\mathop\cup\limits_{i=1}^{d_k}E_{k,i}$.}
\begin{equation*}
\dens(\ME_{k+1}, E_{k,i}):=\dens\Big(\mathop\cup\limits_{j=1}^{d_{k+1}}E_{k+1,j}, E_{k,i}\Big)\geq\rho_{k}.
\end{equation*}
\end{enumerate}
 \end{defi}

Let $\{\ME_k\}^{\infty}_{k=0}$ be a sequence satisfying the nesting conditions. Define $E:=\cap_{k=0}^{\infty}\ME_k$. The following lemma was established in \cite[Proposition 2.1]{McM87}.

\begin{lema}\label{nest}
The density of $E$ in $E_{0,1}$ satisfies
\begin{equation*}
\dens(E,E_{0,1})\geq\prod\limits^{\infty}_{k=0}\rho_{k}.
\end{equation*}
\end{lema}

Now we give the definition of some regions which are needed in the following. For $x>0$, we define
\begin{equation}\label{equ-Rx}
\Lambda(x):=\{z\in \mathbb{C}: |\re z|>x\}.
\end{equation}
For any given $m,n\in \mathbb{Z}$ and $r>0$, we define the closed square by
\begin{equation*}
Q_r^{m,n}:=\{z\in\C: mr\leq \re z\leq (m+1)r \hspace{0.2cm}\text{and}\hspace{0.2cm}  nr\leq \im z \leq (n+1)r \}.
\end{equation*}
Let
\begin{equation}\label{equ-Q-r}
\MQ_r:=\{Q_r^{m,n}:m,n\in\Z\}
\end{equation}
be a partition of $\C$ by the grids with sides of length $r>0$. Sometimes we write $Q_r^{m,n}\in\MQ_r$ as $Q_r$ if we don't want to emphasize the superscript of $Q_r^{m,n}$.

\begin{lema}\label{mQ}
Let $Q\subset\mathbb{C}$ be a square with sides of length $r>0$ and suppose that $f$ is conformal in a neighbourhood of $Q$ with distortion $L(f|_Q)<\infty$. For any $x>0$ and $z_0\in Q$, we have
\begin{equation*}
\textup{Area}\left(\cup\{Q_{r}\in \MQ_r: Q_{r}\cap\left(\partial{f(Q)}\cup(\partial{\Lambda(x)}\cap f(Q))\right)\neq\emptyset\}\right)\leq c r^{2},
\end{equation*}
where $c=16+12\sqrt{2}L(f|_Q)|f'(z_0)|$.
\end{lema}

This lemma was established in \cite[Lemma 2.3]{Sch08} with a different coefficient $c$. For completeness we include a proof here and the argument is slightly different.

\begin{proof}
If $\gamma\subset\mathbb{C}$ is a vertical line with length $l_1>0$, it is clear that
\begin{equation}\label{a}
\#\{Q_{r}\in \MQ_r: Q_{r}\cap\gamma\neq\emptyset\}\leq 4+\frac{2l_{1}}{r}.
\end{equation}
Let $\gamma\subset\mathbb{C}$ be a continuous curve with length $l_{2}=2\sqrt{2}kr>0$, where $k$ is a positive integer. We claim that
\begin{equation}\label{c}
k':=\#\{Q_{r}\in \MQ_r: Q_{r}\cap\gamma\neq\emptyset\}\leq 4+8k.
\end{equation}
Indeed, if $k=1$, then it is easy to see $k'\leq12$. Assume that $k=n$ and in this case $k'\leq4+8n$. If $k=n+1$, let $\gamma(t):[0,1]\to\C$ be a parameterization of $\gamma$ such that the length of $\gamma([0,t_0])$ is $2\sqrt{2}nr$ while the length of $\gamma([t_0,1])$ is $2\sqrt{2}r$, where $0<t_0<1$. Since $\gamma([t_0,1])$ can intersect at most $8$ squares while $\gamma([0,t_0])$ can intersect at most $4+8n$ by the assumption, it follows that $k'\leq 4+8(n+1)$ if $k=n+1$. Hence the claim \eqref{c} is proved.

For the general case, we assume that $\gamma\subset\mathbb{C}$ is a continuous curve with length $l_3>0$. Let $[x]$ be the integer part of $x>0$. By \eqref{c}, we have
\begin{equation}\label{b}
\#\{Q_{r}\in \MQ_r: Q_{r}\cap\gamma\neq\emptyset\}\leq 4+8\left[\frac{l_3}{2\sqrt{2}r}\right]+8\leq 12+\frac{2\sqrt{2}\,l_3}{r}.
\end{equation}

Since $f$ is a conformal map in a neighbourhood of $Q$, we conclude that $\partial f(Q)=f(\partial Q)$. From \eqref{Lf}, the length of $\partial f(Q)$ satisfies
\begin{equation}\label{10}
\begin{split}
 l_4: =\int_{\partial f(Q)}|d\xi|
 =\int_{\partial Q}|f'(z)||dz|
&~ \leq\sup\limits_{z\in Q}|f'(z)|\cdot4r\\
&~ \leq 4\,L(f|_Q)|f'(z_0)|\,r.
\end{split}
\end{equation}
Similarly, the length of $\partial \Lambda(x)\cap f(Q)$ satisfies
\begin{equation}\label{11}
\begin{split}
l_5 \leq 2\,\diam f(Q)
&~\leq 2\sup\limits_{z\in Q}|f'(z)|\cdot\text{diam}(Q)\\
&~\leq 2\sqrt{2}\,L(f|_Q)|f'(z_0)|\,r.
\end{split}
\end{equation}
By \eqref{a}, \eqref{b}, \eqref{10} and \eqref{11}, we have
\begin{align*}
      &\#\{Q_{r}\in \MQ_r: Q_{r}\cap\left(\partial{f(Q)}\cup(\partial{\Lambda(x)}\cap f(Q))\right)\neq\emptyset\} \\
\leq~& \Big(4+\frac{2l_5}{r}\Big)+\Big(12+\frac{2\sqrt{2}\,l_4}{r}\Big)
   = 16+\frac{2l_5+2\sqrt{2}\,l_4}{r}\\
\leq~& 16+12\sqrt{2}\,L(f|_Q)|f'(z_0)|.
\end{align*}
The proof is finished if we notice that the area of each $Q_r$ is $r^2$.
\end{proof}

\subsection{Basic properties of the polynomial and entire function}

For $N\geq 2$, let $P$ be a polynomial with degree at least $2$ which has the form
\begin{equation*}
P(z)=a_0+a_1 z+\cdots+a_N z^N,
\end{equation*}
where $a_i\in\C$ for $0\leq i\leq N$ and $a_0 a_N\neq 0$. In the rest of this article, the polynomial $P$ will be fixed. We denote
\begin{equation}\label{equ-K}
K:=\max\{|a_{0}|, |a_{1}|, \cdots,|a_N|\}>0.
\end{equation}

\begin{lema}\label{pp}
Let $\varepsilon>0$ be any given constant. The following statements hold:
\begin{enumerate}
\item If $|z|\geq 1+\tfrac{K}{\varepsilon\,|a_N|}>1$, then
\begin{equation*}
|P(z)-a_N z^N|\leq\varepsilon\,|a_N|\,|z|^N;
\end{equation*}
\item If $|z|\leq \tfrac{\varepsilon |a_0|}{K+\varepsilon |a_0|}<1$, then
\begin{equation*}
|P(z)-a_0|\leq \varepsilon\,|a_0|.
\end{equation*}
\end{enumerate}
\end{lema}

\begin{proof}
By the definition of $K$ in \eqref{equ-K}, if $|z|\geq 1+\tfrac{K}{\varepsilon\,|a_N|}>1$, then
\begin{equation*}
|P(z)-a_N z^N|\leq K (1+|z|+\cdots+|z|^{N-1})<K\,\frac{|z|^N}{|z|-1}\leq\varepsilon\,|a_N|\,|z|^N.
\end{equation*}
On the other hand, if $|z|\leq \tfrac{\varepsilon |a_0|}{K+\varepsilon |a_0|}<1$, then
\begin{equation*}
|P(z)-a_0|\leq K (|z|+\cdots+|z|^N)<K\,\frac{|z|}{1-|z|}\leq \varepsilon\,|a_0|. \qedhere
\end{equation*}
\end{proof}

Note that
\begin{equation*}
P(z)/z=a_0 z^{-1}+a_1+\cdots+a_N z^{N-1}
\end{equation*}
is a rational function. Let $\D(a,r):=\{z\in\C:|z-a|<r\}$ be the open disk centered at $a\in\C$ with radius $r>0$. For each $R>0$ and $\theta,\xi\in[0,2\pi)$, we denote a closed domain
\begin{equation*}
\mathbb{U}(R,\theta,\xi):=\{z\in\C:|z|\geq R \text{ and } \theta-\tfrac{\xi}{2}\leq \arg(z)\leq \theta+\tfrac{\xi}{2}\}.
\end{equation*}

\begin{lema}\label{lema-univalent}
For every $\theta\in[0,2\pi)$, the rational function $P(z)/z$ is univalent in a neighborhood of $\mathbb{U}(2R_1,\theta,\tfrac{\pi}{N-1})$ and $\overline{\D}(0,R_2/2)$, where
\begin{equation*}
R_1=1+\frac{4K}{|a_N|} \quad\text{and}\quad R_2=\frac{|a_0|}{4K+|a_0|}.
\end{equation*}
\end{lema}

\begin{proof}
(a) If $|z|\geq R_1$, by Lemma \ref{pp}(a) we have
\begin{equation*}
\left|\frac{P(z)}{z}-a_N z^{N-1}\right|\leq \frac{1}{4}\,|a_N|\,|z|^{N-1}.
\end{equation*}
Then one can write $P(z)/z$ as
\begin{equation}\label{equ-P-z}
P_1(z)=\frac{P(z)}{z}=a_N z^{N-1}(1+\varphi(z)),
\end{equation}
where $\varphi(z)$ is holomorphic in $\C\setminus\{0\}$ and $|\varphi(z)|\leq 1/4$ if $|z|\geq R_1$.

Let $w_0\in\C\setminus\{0\}$. For any $w\in\partial \mathbb{U}(|w_0|/2,\arg(w_0),\pi)$, we have
\begin{equation}\label{equ-w-w0}
|w-w_0|>\frac{1}{4}(|w|+|w_0|).
\end{equation}
Let $g(z):=z^{N-1}$. For each $z_0\in\C$ such that $|z_0|\geq 2R_1$, we define $w_0:=g(z_0)=z_0^{N-1}$. Note that $g^{-1}(\mathbb{U}(|w_0|/2,\arg(w_0),\pi))$ consists of $N-1$ disjoint closed domains:
\begin{equation*}
D_k:=\mathbb{U}\left(2^{-1/(N-1)}|z_0|,\arg(z_0)+\frac{2k\pi}{N-1},\frac{\pi}{N-1}\right),
\end{equation*}
where $0\leq k\leq N-2$. Then for $0\leq k\leq N-2$, $z_k:=z_0e^{2k\pi\ii/(N-1)}$ is contained in the interior of $D_k$.

For any $z\in \partial D_k$ with $0\leq k\leq N-2$, we have $z^{N-1}\in\partial \mathbb{U}(|w_0|/2,\arg(w_0),\pi)$. Combining \eqref{equ-P-z} and \eqref{equ-w-w0}, we have
\begin{equation*}
|z^{N-1}-z_0^{N-1}|>\frac{1}{4}(|z|^{N-1}+|z_0|^{N-1})\geq |z^{N-1}\varphi(z)-z_0^{N-1}\varphi(z_0)|.
\end{equation*}
Define $\varphi_1(z):=a_N(z^{N-1}-z_0^{N-1})$ and $\varphi_2(z):=P_1(z)-P_1(z_0)=a_Nz^{N-1}(1+\varphi(z))-a_Nz_0^{N-1}(1+\varphi(z_0))$.
By Rouch\'{e}'s theorem, $\varphi_1(z)=0$ and $\varphi_2(z)=0$ have the same number of roots in each $D_k$, where $0\leq k\leq N-2$. Since $\varphi_1(z)=0$ has exactly one root $z_k$ in each $D_k$, this means that $\varphi_2(z)=0$ has exactly one root in each $D_k$, where $0\leq k\leq N-2$.

On the other hand, \eqref{equ-w-w0} holds also for $w\in \partial \mathbb{U}(|w_0|/2,-\arg(w_0),\pi)$. By Rouch\'{e}'s theorem again, $\varphi_2(z)=0$ has no root in each $-D_k$, where $0\leq k\leq N-2$. By the arbitrariness of $z_0$, it means that $P_1(z)=P(z)/z$ is univalent in a neighborhood of $\mathbb{U}(2R_1,\theta,\tfrac{\pi}{N-1})$, where $\theta\in[0,2\pi)$.

(b) Similarly, by Lemma \ref{pp}(b) one can write $P(z)/z$ as
\begin{equation*}
P_1(z)=\frac{P(z)}{z}=\frac{a_0}{z}(1+\psi(z)),
\end{equation*}
where $\psi(z)$ is holomorphic in $\C$ and $|\psi(z)|\leq 1/4$ if $|z|\leq R_2$. For each $z_0\in\overline{\D}(0,R_2/2)\setminus\{0\}$ and  $z\in\partial \D(0,R_2)$, we have
\begin{equation*}
|z-z_0|>\frac{1}{4}(|z|+|z_0|).
\end{equation*}
Hence
\begin{equation*}
\left|\frac{1}{z}-\frac{1}{z_0}\right|>\frac{1}{4}\,\frac{|z|+|z_0|}{|z z_0|}\geq \left|\frac{\psi(z)}{z}-\frac{\psi(z_0)}{z_0}\right|.
\end{equation*}
Define $\psi_1(z):=a_0(1/z-1/z_0)$ and $\psi_2(z):=P_1(z)-P_1(z_0)=\tfrac{a_0}{z}(1+\psi(z))-\tfrac{a_0}{z_0}(1+\psi(z_0))$.
By Rouch\'{e}'s theorem, $\psi_1(z)=0$ and $\psi_2(z)=0$ have the same number of roots in $\D(0,R_2)$. Since $\psi_1(z)=0$ has exactly one root $z_0$ in $\D(0,R_2)$, this means that $\psi_2(z)=0$ has exactly one root in $\D(0,R_2)$. By the arbitrariness of $z_0$, it means that $P_1(z)=P(z)/z$ is univalent in a neighborhood of $\overline{\D}(0,R_2/2)$.
\end{proof}

Since $P$ is a polynomial, it is easy to see that $P(e^z)/e^z$ is a transcendental entire function. We now give some quantitative estimations on the mapping properties of $f(z)=P(e^z)/e^z$ by applying some properties of $P(z)/z$ obtained above. Recall that $\Lambda(x)=\{z\in\C:|\re z|>x\}$ for $x>0$. We denote
\begin{equation}\label{equ-K1}
K_0:=\min\{|a_0|,\,|a_N|\}>0.
\end{equation}

\begin{cor}\label{cor:x1r1}
Let
\begin{equation}\label{equ-x1r1}
r_0:=\frac{\pi}{N-1} \text{\quad and\quad} R_3:=\log\Big(2+\frac{8K}{K_0}\Big).
\end{equation}
Then for any square $Q\subset \Lambda(R_3)$ with sides of length $r\leq r_0$, the restriction of $f(z)=P(e^z)/e^z$ on a neighbourhood of $Q$ is a conformal map.
\end{cor}

\begin{proof}
We have $|e^z|\geq 2R_1$ if $\re z\geq\log (2R_1)$ and $|e^z|\leq R_2/2$ if $\re z\leq\log (R_2/2)$. Let $Q\subset \Lambda(R_3)$ be a square with sides of length $\pi/(N-1)$. It is easy to see that $\exp$ is injective in a neighbourhood of $Q$ and $\exp(Q)$ is contained in $\overline{\D}(0,R_2/2)$ or $\mathbb{U}(2R_1,\theta,\tfrac{\pi}{N-1})$ for some $\theta\in[0,2\pi)$. This means that $f(z)=P(e^z)/e^z$ is conformal in a neighborhood of $Q$ by Lemma \ref{lema-univalent}.
\end{proof}

We will use the following lemma to estimate $|f'(z)|$ and $|f''(z)/f'(z)|$ for $f(z)=P(e^z)/e^z$.

\begin{lema}\label{lema-est-p1}
Suppose that $|z|\geq R_4$ or $|z|\leq R_5$, where
\begin{equation*}
R_4=1+\max\Big\{\tfrac{2K+4}{|a_N|},\tfrac{K}{|a_N|}\big(\tfrac{2N^2}{N-1}+1\big)\Big\}
\text{ and }
R_5=\min\Big\{\tfrac{|a_0|}{2(KN+2)},\tfrac{1}{2N}\sqrt{\tfrac{|a_0|}{K}}\Big\}.
\end{equation*}
Then
\begin{equation*}
\left|P'(z)-\frac{P(z)}{z}\right|> 2
\text{\quad and\quad}
\left|\frac{z^{2}P''(z)}{z P'(z)-P(z)}-1 \right|<N.
\end{equation*}
\end{lema}

\begin{proof}
A direct calculation shows that
\begin{equation*}
P'(z)=\sum_{k=1}^N k a_k z^{k-1} \text{\quad and\quad} P''(z)=\sum_{k=2}^N k(k-1) a_k z^{k-2}.
\end{equation*}
This means that
\begin{equation}\label{equ-1}
P'(z)-\frac{P(z)}{z}=\sum_{k=1}^N k a_k z^{k-1}-\sum_{k=0}^N a_k z^{k-1}=\sum_{k=0}^N (k-1) a_k z^{k-1}
\end{equation}
and
\begin{equation}\label{equ-2}
\frac{z^{2}P''(z)}{z P'(z)-P(z)}-1=\frac{\sum_{k=0}^N k(k-1) a_k z^k}{\sum_{k=0}^N (k-1) a_k z^k}-1
=\frac{\sum_{k=0}^N (k-1)^2 a_k z^k}{\sum_{k=0}^N (k-1) a_k z^k}.
\end{equation}

If $|z|\geq 1+\tfrac{2K+4}{|a_N|}>3$, by \eqref{equ-1} we have
\begin{equation}\label{equ-est-f-1}
\begin{split}
\left|P'(z)-\frac{P(z)}{z}\right|
\geq~& |a_N|\,(N-1)|z|^{N-1}-K(N-1)(|z|^{N-2}+\cdots+|z|+1)\\
\geq~& (N-1)|z|^{N-1}\Big(|a_N|-\frac{K}{|z|-1}\Big)\\
\geq~& \frac{|a_N|}{2}|z|^{N-1}\geq \frac{|a_N|}{2}|z|>2.
\end{split}
\end{equation}
If $|z|\leq \tfrac{|a_0|}{2(KN+2)}<\frac{1}{2}$, we have
\begin{equation}\label{equ-est-f-2}
\begin{split}
\left|P'(z)-\frac{P(z)}{z}\right|
\geq~& \frac{|a_0|}{|z|}-K(N-1)(|z|+\cdots+|z|^{N-1})\\
\geq~& \frac{|a_0|}{|z|}-K(N-1)>\frac{|a_0|}{2|z|}\geq KN+2>2.
\end{split}
\end{equation}

For the second inequality, if $|z|\geq 1+\tfrac{K}{|a_N|}\big(\tfrac{2N^2}{N-1}+1\big)>8$, by \eqref{equ-2} we have
\begin{equation*}
\begin{split}
\left|\frac{z^{2}P''(z)}{z P'(z)-P(z)}-1\right|
\leq~& N-1+\left|\frac{\sum_{k=0}^{N-1} (k-1)(N-k) a_k z^k}{\sum_{k=0}^N (k-1) a_k z^k}\right|\\
\leq~& N-1+\frac{K N^2 }{N-1}\cdot\frac{|z|+\cdots+|z|^{N-1}}{|a_N|\,|z|^N-K(|z|+\cdots+|z|^{N-1})}\\
\leq~& N-1+\frac{K N^2 }{N-1}\cdot\frac{1}{|a_N|(|z|-1)-K}\leq N-\frac{1}{2}<N.
\end{split}
\end{equation*}
If $|z|\leq \tfrac{1}{2N}\sqrt{|a_0|/K}<\frac{1}{2}$, by \eqref{equ-2} we have
\begin{equation*}
\begin{split}
\left|\frac{z^{2}P''(z)}{z P'(z)-P(z)}-1\right|
\leq~& 1+\left|\frac{\sum_{k=2}^N k(k-1) a_k z^k}{\sum_{k=0}^N (k-1) a_k z^k}\right|\\
\leq~& 1+\frac{K N^2(|z|^2+\cdots+|z|^N)}{|a_0|-KN(|z|^2+\cdots+|z|^N)}\\
\leq~& 1+\frac{2K N^2|z|^2}{|a_0|-2KN\,|z|^2}\leq 1+\frac{N}{2N-1}\leq \frac{5}{3}< N.\qedhere
\end{split}
\end{equation*}
\end{proof}

\begin{cor}\label{cor:x2m1}
Let
\begin{equation}\label{equ-x2}
R_6:=\max\big\{\log R_4,-\log R_5\big\}.
\end{equation}
Then for any $z\in \Lambda(R_6)$, the function $f(z)=P(e^z)/e^z$ satisfies
\begin{equation*}
|f'(z)|>2 \text{\quad and \quad}  \frac{|f''(z)|}{|f'(z)|}<N.
\end{equation*}
\end{cor}

\begin{proof}
Denote $P_1(w):=P(w)/w$. Therefore, $f(z)=P(e^z)/e^z=P_1\circ\exp(z)$. It is easy to check that
\begin{equation*}
f'(z)=P_1'(e^z)e^z \text{\quad and \quad} f''(z)=P_1''(e^z)e^{2z}+P_1'(e^z)e^z.
\end{equation*}
Let $w=e^z$. By a straightforward computation, we have
\begin{equation}\label{equ-f-w-z}
f'(z)=P_1'(w)w=P'(w)-\frac{P(w)}{w}
\end{equation}
and
\begin{equation*}
 \frac{f''(z)}{f'(z)}=\frac{P_1''(w)w^2+P_1'(w)w}{P_1'(w)w}=\frac{w^{2}P''(w)}{w P'(w)-P(w)}-1.
\end{equation*}
Then the result follows from Lemma \ref{lema-est-p1} immediately.
\end{proof}

\subsection{Escaping and fast escaping sets}

Let $f$ be a transcendental entire function. A point $a\in\mathbb{C}$ is called an \textit{asymptotic value} of $f$ if there exists a continuous curve $\gamma(t)\subset\mathbb{C}$ with $0<t<\infty$, such that $\gamma(t)\rightarrow\infty$ as $t\rightarrow \infty$ and $f(\gamma(t)) \rightarrow a$ as $t\rightarrow \infty$.

\begin{lema}\label{asym}
The entire function $f(z)=P(e^z)/e^z$ does not have any finite asymptotic value.
\end{lema}
\begin{proof}
Assume that $a\in\C$ is a finite asymptotic value of $f(z)$. Then by definition, there exists a continuous curve $\gamma(t)\subset\mathbb{C}$ with $0<t<\infty$, such that $\gamma(t)\rightarrow\infty$ as $t\rightarrow\infty$ and $f(\gamma(t))\rightarrow a$ as $t\rightarrow\infty$. This means that
\begin{equation*}
\lim\limits_{t\rightarrow\infty}\frac{P(w)}{w}\circ e^{\gamma(t)} =a.
\end{equation*}
Denote $\gamma(t)=x(t)+\ii y(t)$ and let $w_{1}$, $w_{2}$, $\cdots$, $w_N$ be the $N$ roots of the equation $P(w)=aw$. We define the set $Y:=\{\arg w_{i}+2k\pi: 1\leq i\leq N, \,k\in\mathbb{Z}\}$. If $x(t)$ is unbounded as $t\rightarrow\infty$, then $f(\gamma(t))$ is also unbounded and this is a contradiction. Hence $|x(t)|\leq A$ for some constant $A>0$ for all $t$. Since $\gamma(t)\rightarrow\infty$ as $t\rightarrow\infty$, this implies that $y(t)\rightarrow\infty$ as $t\rightarrow\infty$. Therefore, for each $y_{0}\in\mathbb{R}\setminus Y$, there exists a sequence $\{z_{n}\}\subset\gamma(t)$ such that $\im z_{n}\rightarrow\infty$ as $n\rightarrow\infty$ and $\lim_{n\rightarrow\infty}e^{\ii\textup{Im} z_{n}}=e^{\ii y_{0}}$. Since $|x(t)|\leq A$, it follows that $\lim_{t\rightarrow\infty}e^{x(t)}\neq 0$. This implies that $\lim_{n\rightarrow\infty}f(z_n)=\lim_{n\rightarrow\infty}P(e^{z_n})/e^{z_n}\neq a$, which is a contradiction.
\end{proof}

Let $f$ be a transcendental entire function. The set
\begin{equation}\label{equ-I-f}
I(f):= \{z\in\C: f^{\circ n}(z)\rightarrow \infty \text{ as } n\to\infty\}
\end{equation}
is called the \textit{escaping set} of $f$. We use $\sing(f^{-1})$ to denote the set of \textit{singular values} of $f$ which consists of all the critical values and asymptotic values of $f$ and their accumulation points.

\begin{cor}\label{cor-I-J}
The escaping set $I(f)$ of $f(z)=P(e^z)/e^z$ is contained in the Julia set $J(f)$.
\end{cor}

\begin{proof}
It is clear that the set of the critical values of $f(z)=P(e^z)/e^z$ is finite. From Lemma \ref{asym}, it follows that $\sing(f^{-1})$ is bounded. According to \cite[Theorem 1]{EL92}, we have $I(f)\subset J(f)$.
\end{proof}

Actually, we will estimate the area of the complement of the fast escaping set in next section. Let $f$ be a transcendental entire function. The \textit{maximal modulus function} is defined by
\begin{equation*}
M(r,f):=\max_{|z|=r}|f(z)|, \text{ where } r>0.
\end{equation*}
We use $M^{\circ n}(r,f)$ to denote the $n$-th iterate of $M(r,f)$ with respect to the variable $r>0$, where $n\in\N$. The notation $M(r,f)$ is written as $M(r)$ if the function $f$ is known clearly. A subset of the escaping set, called the \textit{fast escaping set} $A(f)$ was introduced in \cite{BH99} and can be defined \cite{RS12} by
\begin{equation}\label{equ-A-f}
A(f):=\{z: \text{ there is } \ell\in\N \text{ such that } |f^{\circ(n+\ell)}(z)|\geq M^{\circ n}(R) \text{ for }n\in\N\}.
\end{equation}
Here $R>0$ is a constant such that $M^{\circ n}(R)\to\infty$ as $n\to\infty$. It is proved in \cite[Theorem 2.2(b)]{RS12} that $A(f)$ is independent of the choice of $R$ such that $M^{\circ n}(R)\to\infty$ as $n\to\infty$.

\begin{lema}\label{lema-fast}
Let $R>0$ be a constant and define $u_0:=R$. For $n\geq 1$, define $u_n$ inductively by $u_n:=R e^{Ru_{n-1}}$. Let $v_0\in\R$ and define $v_n$ inductively by $v_n:=e^{v_{n-1}}$ for $n\geq 1$. Then there is $\ell\in\N$ such that $v_{n+\ell}\geq 2R u_n$ for all $n\in\N$.
\end{lema}

\begin{proof}
For any $v_0\in\R$, there exists an integer $\ell\in\N$ such that $v_\ell\geq 2R^2$. Shifting the subscript of $(v_n)_{n\in\N}$ if necessary, it is sufficient to prove that if $v_0\geq 2R^2$, then $v_n\geq 2R u_n$ for all $n\in\N$. Suppose that $v_{n-1}\geq 2Ru_{n-1}$ for some $n\geq 1$ (note that $v_0\geq 2R u_0$). We hope to obtain that $v_n\geq 2R u_n$. Note that $v_n=e^{v_{n-1}}\geq e^{2R u_{n-1}}$ and $u_n=R e^{R u_{n-1}}$. It is sufficient to obtain $R u_{n-1}\geq \log(2R^2)$. This is true since $u_{n-1}\geq R$ and $R^2\geq\log(2R^2)$ for all $R>0$.
\end{proof}

\begin{cor}\label{cor-fast-escaping}
Let $z_0\in\C$ and suppose that $z_n=f^{\circ n}(z_0)$ satisfies $|z_n|\geq \xi_n$ for all $n\in\N$, where $\xi_n>0$ is defined inductively by
\begin{equation*}
\xi_n=2\exp(\xi_{n-1}/2) \text{ with } \xi_0>0.
\end{equation*}
Then $z_0$ is contained in the fast escaping set of $f(z)=P(e^z)/e^z$.
\end{cor}

\begin{proof}
Recall that $N\geq 2$ is the degree of the polynomial $P$ and $K>0$ is defined in \eqref{equ-K}. According to Lemma \ref{pp}, there exists $\delta_0\geq 1$ such that if $\delta\geq\delta_0$, then the maximal modulus function of $f$ satisfies
\begin{equation*}
M(\delta)=M(\delta,f)\leq 2K e^{(N-1)\delta}.
\end{equation*}
On the other hand, there exists $\delta_1>0$ such that for all $\delta\geq\delta_1$, then $M^{\circ n}(\delta)$ is monotonically increasing as $n$ increases. Since the Julia set of $f$ is non-empty, this means that $M^{\circ n}(\delta)\to\infty$ as $n\to\infty$ if $\delta\geq\delta_1$.

Define
\begin{equation*}
R:=\max\{2K,(N-1)\delta_0,\delta_1\}\geq 1.
\end{equation*}
We denote $u_0=R$ and for $n\geq 1$, define $u_n$ inductively by $u_n=R e^{Ru_{n-1}}$. Then we have $M^{\circ n}(R)\leq u_n$ for all $n\in\N$. By the definition of $\xi_n$, we have $\xi_n=2\exp^{\circ n}(\xi_0/2)$. Let $v_0:=\xi_0/2$ and define $v_n:=e^{v_{n-1}}$ for $n\geq 1$. According to Lemma \ref{lema-fast}, there exists $\ell\in\N$ such that for all $n\in\N$,
\begin{equation*}
|f^{\circ (n+\ell)}(z_0)|=|z_{n+\ell}|\geq \xi_{n+\ell}=2v_{n+\ell}\geq 4R u_n\geq u_n\geq M^{\circ n}(R).
\end{equation*}
By the defintion of $R$, we have $M^{\circ n}(R)\to\infty$ as $n\to\infty$. This means that $z_0$ is contained in the fast escaping set of $f$.
\end{proof}

\section{Proof of the theorems}

\subsection{Proof of Theorem \ref{thm}}

Recall that $N\geq 2$ is the degree of the polynomial $P$. Let $r>0$ be fixed such that
\begin{equation}\label{equ-r-x}
r\leq\frac{1}{4N}.
\end{equation}
We define
\begin{equation}\label{equ-r-x-1}
x':=\max\{R_3, R_6, 6\log 2\},
\end{equation}
where $R_3$ and $R_6$ are constants introduced in Corollary \ref{cor:x1r1} and Corollary \ref{cor:x2m1} respectively.

Recall that $\Lambda(x)=\{z\in\C:|\re z|>x\}$ is the set defined in \eqref{equ-Rx} for all $x>0$. Let $Q_{0}$ be a square in $\Lambda(x)$ with sides of length $r$, where $x\geq x'$. Since $r<r_0=\pi/(N-1)$, from Corollary \ref{cor:x1r1} we know that $f$ is conformal in a neighbourhood of $Q_{0}$. For $k\in\N$, define
\begin{equation}\label{equ-x-k}
x_{k}:=2\exp^{\circ k}(x/2).
\end{equation}
In particular, $x_0=x\geq x'$ and we have $x_{k+1}=2\exp(x_k/2)>x_k\geq x'$ since $2e^{x/2}>x$ for all $x\in\R$.
Recall that $\MQ_r$ is a collection of grids with sides of length $r>0$ defined in \eqref{equ-Q-r}. For any subset $E$ of $Q_{0}$ in $\Lambda(x_{0})$ and $k\in\N$, define
\begin{equation*}
\text{pack}(f^{\circ k}(E)):=\{Q_{r}\in \MQ_r: Q_{r}\subset f^{\circ k}(E)\cap \Lambda(x_{k})\}.
\end{equation*}

We now define a sequence of families of measurable sets satisfying the nesting conditions based on the square $Q_0$. Let $\ME_0:=\{Q_{0}\}$ and for $k\geq 1$, define inductively
\begin{equation*}
\ME_k:=\{F_{k}\subset Q_{0}: F_{k}\subset E_{k-1}\in \ME_{k-1} \text{ and }  f^{\circ k}(F_{k})\in \text{pack}(f^{\circ k}(E_{k-1})) \}.
\end{equation*}
It is clear that $\ME_k$ is a finite collection of measurable subsets of $\mathbb{C}$ for all $k\in \mathbb{N}$. Denote the elements of $\ME_k$ by $E_{k,i}$, where $1\leq i\leq d_k$.

By definition, for all $k\in \mathbb{N},$ we have $f^{\circ(k+1)}(E_{k,i})=f( Q_{r}^k)$, where\footnote{Note that $Q_r^k\subset \Lambda(x_k)$ is a square depending also on the subscript `$i$' of $E_{k,i}$, where $k\in\N$ and $1\leq i\leq d_k$. We omit this index here for simplicity.} $Q_r^k$ is a square with sides of length $r$ and $Q_r^k\subset \Lambda(x_{k})$. From \eqref{equ-nonlinearity}, Corollary \ref{cor:x2m1} and \eqref{equ-r-x}, we have
\begin{equation*}
N(f|_{Q_r^k})< N\sqrt{2}r\leq \frac{\sqrt{2}}{4}.
\end{equation*}
By Lemma \ref{LN}, the distortion of $f$ on $Q_r^k$ satisfies
\begin{equation}\label{L2}
L(f|_{Q_r^k})\leq 1+2 N(f|_{Q_r^k})<2.
\end{equation}
For every $k\in\N$, let $z_k$ be any point in $Q_r^k\subset \Lambda(x_{k})$. From \eqref{Lf} and \eqref{L2} we have
\begin{equation}\label{mfQ}
\begin{split}
\Area(f(Q_{r}^{k}))
=~& \int_{Q_r^k}|f'(z)|^{2}dxdy \geq\inf\limits_{z\in Q_r^k}|f'(z)|^{2}\cdot\Area(Q_r^k)\\
\geq~& \frac{|f'(z_k)|^{2}}{(L(f|_{ Q_r^k}))^{2}}\cdot r^{2}>\frac{1}{4}|f'(z_k)|^{2}r^{2}
\end{split}
\end{equation}
and
\begin{equation}\label{mfQ-1}
\begin{split}
\text{diam}(f( Q_r^k))
&~\leq \sup\limits_{z\in Q_r^k}|f'(z)|\cdot\text{diam}(Q_r^k)\\
&~\leq L(f|_{ Q_r^k})|f'(z_k)|\cdot\sqrt{2}r<2\sqrt{2}|f'(z_k)|r.
\end{split}
\end{equation}
Recall that $K_0=\min\{|a_0|,\,|a_N|\}>0$ is the constant defined in \eqref{equ-K1}. By \eqref{equ-est-f-1}, \eqref{equ-est-f-2} and \eqref{equ-f-w-z}, we have
\begin{equation}\label{equ-f-deri}
|f'(z_k)|>\frac{1}{2}\,K_0 e^{|\re z_k|}>\frac{1}{2}\,K_0e^{x_k}.
\end{equation}

For $k\in \mathbb{N}$ and $1\leq i\leq d_k$, we denote
\begin{equation*}
B_1:=\cup \{Q_{r}\in \MQ_r:  Q_r\subset  f^{\circ(k+1)}(E_{k,i})\cap (\C\setminus \Lambda(x_{k+1}))\}
\end{equation*}
and
\begin{equation*}
B_2:=\cup\{Q_{r}\in \MQ_r: Q_{r}\cap(\partial f^{\circ(k+1)}(E_{k,i})\cup(\partial \Lambda(x_{k+1})\cap f^{\circ(k+1)}(E_{k,i})))\neq\emptyset\}.
\end{equation*}
Recall that $f^{\circ(k+1)}(E_{k,i})=f(Q_{r}^k)$ for some square $Q_r^k$ in $\Lambda(x_{k})$ with sides of length $r$, where $k\in \mathbb{N}$ and $1\leq i\leq d_k$.
From \eqref{mfQ}, \eqref{mfQ-1} and \eqref{equ-f-deri}, we have
\begin{equation}\label{one}
\begin{split}
 &~\frac{\Area\left(B_1\right)}{\Area(f^{\circ(k+1)}(E_{k,i}))} \leq\frac{2x_{k+1}\text{diam}(f^{\circ(k+1)}(E_{k,i}))}{\Area(f^{\circ{(k+1)}}(E_{k,i}))}\\
=&~ \frac{2x_{k+1}\text{diam}( f(Q_r^k))}{\Area(f( Q_r^k))}
< \frac{16\sqrt{2}x_{k+1}}{|f'(z_k)|\,r}<\frac{32\sqrt{2}}{K_0 r}\cdot\frac{x_{k+1}}{e^{x_{k}}}.
\end{split}
\end{equation}
Note that $x_{k+1}\geq x_1=2e^{x/2}$ for all $k\in\N$ and $x\geq 6\log 2$ by \eqref{equ-r-x-1}. By Lemma \ref{mQ}, \eqref{L2}, \eqref{mfQ} and \eqref{equ-f-deri}, we have
\begin{equation}\label{two}
\begin{split}
&~\frac{\Area(B_2)}{\Area(f^{k+1}(E_{k,i}))} \leq\frac{\left(16+12\sqrt{2}L(f|_{Q_r^k})|f'(z_k)|\right)r^{2}}{\Area(f(Q_r^k))}\\
<&~\frac{32(2+3\sqrt{2}|f'(z_k)|)}{|f'(z_k)|^{2}}
<\frac{256}{K_0^2 e^{2x_k}}+\frac{192\sqrt{2}}{K_0 e^{x_k}}\\
\leq &~\Big(\frac{128}{K_0^2}\cdot\frac{1}{e^{3x/2}}+\frac{96\sqrt{2}}{K_0}\cdot\frac{1}{e^{x/2}}\Big)\cdot\frac{x_{k+1}}{e^{x_{k}}}
\leq\Big(\frac{1}{4K_0^2}+\frac{12\sqrt{2}}{K_0}\Big)\cdot\frac{x_{k+1}}{e^{x_{k}}}.
\end{split}
\end{equation}
For all $k\in \mathbb{N}$ and $1\leq i\leq d_k$, by \eqref{one} and \eqref{two}, we have
\begin{equation}\label{three}
\begin{split}
    &~\dens\left(\bigcup \text{pack}(f^{\circ (k+1)}(E_{k,i})), f^{\circ (k+1)}(E_{k,i})\right)\\
\geq&~\frac{\Area\left(\bigcup \{Q_{r}\in \MQ_r:  Q_{r}\cap f^{\circ (k+1)}(E_{k,i})\neq\emptyset\right)}{\Area(f^{\circ (k+1)}(E_{k,i}))}
   -\frac{\Area(B_1)+\Area(B_2)}{\Area(f^{\circ (k+1)}(E_{k,i}))} \\
>&~1-c_0\,\frac{x_{k+1}}{e^{x_{k}}}\geq 1-c_1\,\frac{x_{k+1}}{e^{x_{k}}},
\end{split}
\end{equation}
where
\begin{equation}\label{c2}
c_1\geq c_0:=\frac{32\sqrt{2}}{K_0 r}+\frac{1}{4K_0^2}+\frac{12\sqrt{2}}{K_0}.
\end{equation}
Comparing \eqref{equ-r-x-1}, we assume that $x^*>0$ is a fixed constant such that
\begin{equation}\label{equ-r-x-1-new}
x^{*}\geq\max\{R_3, R_6, 6\log 2,12+2\log c_1\}.
\end{equation}
Moreover, we suppose that the sequence $\{x_k\}_{k\in\N}$ in \eqref{equ-x-k} is chosen such that the initial point satisfies $x_0= x\geq x^*$. Then, all the statements above are still true since $x^*\geq x'$.

By a straightforward induction, one can show that for all $k\in\N$ and $x\in\R$,
\begin{equation*}
\exp^{\circ (k+1)}(x)\geq \exp(k)\exp(x).
\end{equation*}
Since $x_{k+1}=2e^{x_k/2}$, we have
\begin{equation}\label{equ-x-k-1}
\frac{x_{k+1}}{e^{x_k}}=\frac{2}{e^{x_k/2}}=\frac{2}{\exp^{\circ(k+1)}(x/2)}\leq \frac{2}{e^k}\cdot\frac{1}{e^{x/2}}.
\end{equation}
On the other hand, by \eqref{equ-r-x-1-new}, we have $e^{x/2}\geq c_1 e^6>6c_1 e^4$ since $x\geq x_*$. Therefore,
\begin{equation}\label{equ-c-1-exp}
c_1 e^4 \,\frac{x_{k+1}}{e^{x_{k}}}\leq c_1 e^4\cdot\frac{2}{e^k}\cdot\frac{1}{e^{x/2}}\leq c_1 e^4\cdot\frac{2}{e^{x/2}}<\frac{1}{3}.
\end{equation}

Define $V:=f(Q_r^k)$ and let $G:=f^{-(k+1)}: V\to Q_{0}$ be the inverse of $f^{\circ (k+1)}|_{E_{k,i}}$, where $k\in\N$ and $1\leq i\leq d_k$. By Lemma \ref{distortion}, Corollary \ref{cor:x2m1} and \eqref{equ-r-x}, the distortion of $G$ on $V$ satisfies
\begin{equation}\label{G}
 L(G|_V)<\exp(\tfrac{2}{2-1})=e^2.
\end{equation}
From \eqref{density} and \eqref{G}, we have
\begin{equation*}
\begin{split}
  &~\dens\left(\ME_{k+1}, E_{k,i}\right)= 1-\dens\left(E_{k,i}\setminus \ME_{k+1}, E_{k,i}\right) \\
=&~ 1-\dens\Big(G\big(f^{\circ (k+1)}(E_{k,i}\setminus \ME_{k+1})\big), G\big(f^{\circ (k+1)}(E_{k,i})\big)\Big)\\
\geq&~ 1-L(G|_V)^{2}\,\dens\left(f^{\circ (k+1)}(E_{k,i})\setminus \bigcup\text{pack}(f^{\circ (k+1)}(E_{k,i}), f^{\circ (k+1)}(E_{k,i})\right)\\
\geq&~ 1-e^4\left(1-\dens\big(\bigcup \text{pack}(f^{\circ (k+1)}(E_{k,i})), f^{\circ (k+1)}(E_{k,i})\big)\right).
\end{split}
\end{equation*}
Therefore, by \eqref{three} and \eqref{equ-c-1-exp}, we have
\begin{equation}
\dens(\ME_{k+1}, E_{k,i})\geq 1-c_1 e^4 \,\frac{x_{k+1}}{e^{x_{k}}}\geq\frac{2}{3},
\end{equation}
where $k\in\N$ and $1\leq i\leq d_k$. For all $k\in\N$, by setting
\begin{equation}\label{equ-rho-k}
\rho_{k}:=1-c_1 e^4 \,\frac{x_{k+1}}{e^{x_{k}}},
\end{equation}
it is easy to see that $\{\ME_k\}_{k=0}^{\infty}$ satisfies the nesting conditions.

Define $E=\cap_{k=0}^{\infty}\ME_k$. Recall that $A(f)$ is the fast escaping set of $f$ defined in \eqref{equ-A-f}. Since every point $z\in E_{k,i}$ satisfies $f^{\circ j}(z)\in \Lambda(x_j)$ for $0\leq j\leq k$ and $x_k\to+\infty$ as $k\to\infty$, it means that $E$ is contained in the fast escaping set $A(f)$ by \eqref{equ-x-k} and Corollary \ref{cor-fast-escaping}. According to Lemma \ref{nest}, we have
\begin{equation*}
\dens(A(f), Q_{0})\geq \dens(E, Q_{0})\geq\prod\limits_{k=0}^{\infty}\rho_{k}.
\end{equation*}
Note that $\log(1-t)>-2t$ for $t\in(0,1/2)$. By \eqref{equ-c-1-exp} and \eqref{equ-rho-k} we have
\begin{equation*}
\begin{split}
\log \Big(\prod_{k=0}^{\infty}\rho_{k}\Big)
=&~\sum_{k=0}^\infty\log\Big(1-c_1 e^4 \,\frac{x_{k+1}}{e^{x_{k}}}\Big)\geq -2\sum_{k=0}^\infty c_1 e^4 \,\frac{x_{k+1}}{e^{x_{k}}} \\
\geq &~ -\frac{4 c_1 e^4}{e^{x/2}}\sum_{k=0}^\infty \frac{1}{e^k}>-\frac{8 c_1 e^4}{e^{x/2}}.
\end{split}
\end{equation*}
Since $e^{-t}\geq 1-t$ for all $t\in\R$, we have
\begin{equation}\label{use}
\dens(A(f), Q_{0})>\exp{\left(-\frac{8 c_1 e^4 }{e^{x/2}}\right)}\geq 1-\frac{8 c_1 e^4 }{e^{x/2}}
\end{equation}
for all $x\geq x^{*}$ and all square $Q_{0}\subset \Lambda(x)$ with sides of length $r$.

\begin{thm}\label{thm-main-restate}
Let $S$ be any horizontal strip of width $2\pi$. Then the area of the complement of the fast escaping set of $f(z)=P(e^z)/e^z$ satisfies
\begin{equation}\label{equ-area-bound}
\Area(S\cap A(f)^c)\leq (4\pi+4r)\left(x^{*}+r+8c_1\, e^{4-x^*/2}\frac{r}{1-e^{-r/2}}\right)<\infty,
\end{equation}
where $r$, $c_1$ and $x^*$ are any positive constants satisfying \eqref{equ-r-x}, \eqref{c2} and \eqref{equ-r-x-1-new} respectively.
\end{thm}

\begin{proof}
Define the half strip $S_+$ by
\begin{equation*}
S_+:=\{z\in\mathbb{C}: 0\leq \im z\leq2\pi \text{ and } \re z\geq 0\}.
\end{equation*}
We take
\begin{equation}\label{equ-m-0-n-0}
m_0=[x^{*}/r]+1 \text{ and } n_0=[2\pi/r]+1,
\end{equation}
where $[x]$ denotes the integer part of $x\geq 0$. Recall that $Q_r^{m,n}$ is defined as
\begin{equation*}
Q_r^{m,n}:=\{z\in\C: mr\leq \re z\leq (m+1)r \hspace{0.2cm}\text{and}\hspace{0.2cm}  nr\leq \im z \leq (n+1)r \},
\end{equation*}
where $m,n\in\Z$. Since $Q_r^{m,n}\subset \Lambda(x^{*})$ for all $m\geq m_{0}$, we get
\begin{equation}\label{3.11}
\dens(A(f), Q_r^{m,n})>1-\frac{8 c_1 e^4 }{\exp(mr/2)}
\end{equation}
for all $m\geq m_{0}$ by \eqref{use}. So
\begin{align*}
\Area( S_+\cap A(f)^c)&\leq \Area\left(\left(\mathop\cup\limits_{m=0}^{\infty} \mathop\cup\limits_{n=0}^{n_{0}}Q_r^{m,n}\right)\setminus A(f)\right)\\
               &\leq\sum\limits_{m=0}^{\infty}\sum\limits_{n=0}^{n_{0}}\Area(Q_r^{m,n}\setminus A(f))\\
               &\leq\sum\limits_{m=0}^{\infty}\sum\limits_{n=0}^{n_{0}}(1-\dens(A(f), Q_r^{m,n}))\cdot\Area(Q_r^{m,n}).
\end{align*}
By \eqref{equ-m-0-n-0} and \eqref{3.11}, we obtain
\begin{equation*}
\begin{split}
\Area( S_+\cap A(f)^c)
\leq &~ r^2\Big(\sum_{m=0}^{m_0-1}\sum_{n=0}^{n_{0}}1+\sum_{m=m_0}^{\infty}\sum_{n=0}^{n_{0}}\frac{8 c_1 e^4 }{\exp(mr/2)}\Big) \\
\leq &~(2\pi+2r)\left(x^{*}+r+8c_1\, e^{4-x^*/2}\frac{r}{1-e^{-r/2}}\right).
\end{split}
\end{equation*}
This means that $\Area( S_+\cap A(f)^c)<\infty$ for every fixed $r>0$ satisfying \eqref{equ-r-x}.
Similarly, one can obtain
\begin{equation*}
\Area( S_-\cap A(f)^c)\leq (2\pi+2r)\left(x^{*}+r+8c_1\, e^{4-x^*/2}\frac{r}{1-e^{-r/2}}\right),
\end{equation*}
where $S_-=\{z\in\mathbb{C}: 0\leq \im z\leq2\pi \text{ and } \re z\leq 0\}$. Since $f(z)=f(z+2\pi\ii)$, for any horizontal strip $S$ of width $2\pi$, we have
\begin{equation*}
\Area(S\cap A(f)^c)\leq (4\pi+4r)\left(x^{*}+r+8c_1\, e^{4-x^*/2}\frac{r}{1-e^{-r/2}}\right).
\end{equation*}
This completes the proof of Theorem \ref{thm-main-restate} and hence Theorem \ref{thm}.
\end{proof}

\subsection{Proof of Theorem \ref{thm-sin}}

Consider the quadratic polynomial
\begin{equation*}
P(z)=\frac{\alpha}{2} z^2+\ii\beta z-\frac{\alpha}{2}, \text{ where } \alpha\neq 0 \text{ and }\beta\in\C.
\end{equation*}
We then have
\begin{equation*}
f(z):=\frac{P(e^z)}{e^z}=\frac{\alpha}{2} e^z+\ii\beta-\frac{\alpha}{2} e^{-z}.
\end{equation*}
Note that $\alpha\sin(z+\beta)$ is conjugated by $z\mapsto \ii(z+\beta)$ to $f(z)$,
In order to prove Theorem \ref{thm-sin}, it is sufficient to prove the corresponding statements on $f$.

Now we collect all the needing constants in the proof. Note that the degree of $P$ is $\deg(P)=N=2$. By \eqref{equ-r-x} we fix the choice of $r>0$ by setting
\begin{equation*}
r=1/8.
\end{equation*}
By \eqref{equ-K1}, we have $K_0=|\alpha|/2$. From \eqref{c2}, we fix
\begin{equation*}
c_1=c_0=\frac{536\sqrt{2}}{|\alpha|}+\frac{1}{|\alpha|^2}.
\end{equation*}
By \eqref{equ-x1r1}, we have
\begin{equation*}
R_3=\log\Big(2+\frac{16K}{|\alpha|}\Big), \text{ where } K=\max\{|\alpha|/2,|\beta|\}.
\end{equation*}
According to Lemma \ref{lema-est-p1}, we have
\begin{equation*}
R_4=\max\Big\{1+\frac{4(K+2)}{|\alpha|},1+\frac{18K}{|\alpha|}\Big\} \text{\quad and\quad}
R_5=\min\Big\{\frac{|\alpha|}{8(K+1)},\frac{1}{4}\sqrt{\frac{|\alpha|}{2K}}\Big\}.
\end{equation*}

Since $K\geq |\alpha|/2>0$, we have
\begin{equation*}
\frac{8(K+1)}{|\alpha|}> \frac{8K}{|\alpha|}\geq 4\sqrt{\frac{2K}{|\alpha|}}, \quad
\frac{8(K+1)}{|\alpha|}=\frac{4K}{|\alpha|}+\frac{4(K+2)}{|\alpha|}>1+\frac{4(K+2)}{|\alpha|}
\end{equation*}
and
\begin{equation*}
1+\frac{18K}{|\alpha|}= 1+\frac{16K}{|\alpha|}+\frac{K}{|\alpha|/2}\geq 2+\frac{16K}{|\alpha|}.
\end{equation*}
Hence by \eqref{equ-r-x-1-new}, we can fix
\begin{equation*}
x^*=\max\Big\{\log\Big(1+\frac{18K}{|\alpha|}\Big),\log\Big(\frac{8(K+1)}{|\alpha|}\Big), 6\log 2,12+2\log c_1\Big\}.
\end{equation*}
By Theorem \ref{thm-main-restate}, the proof of Theorem \ref{thm-sin} is finished module the statement on the sine and cosine functions.

\vskip0.2cm

Let $S$ be a vertical strip with width $2\pi$. If $\alpha=1$ and $\beta=0$, then $K=1/2$ and
\begin{equation}\label{equ-cst-rcx}
r=1/8, \quad c_1=536\sqrt{2}+1 \text{\quad and\quad} x^*=12+2\log\big(536\sqrt{2}+1\big).
\end{equation}
From \eqref{equ-area-bound} we have
\begin{equation*}
\begin{split}
 &~\Area(S\cap A(\sin z)^c) \\
\leq&~\Big(4\pi+\frac{1}{2}\Big)\Big(\frac{97}{8}+2\log(536\sqrt{2}+1)+\frac{1}{e^2-e^{31/16}}\Big)<361.
\end{split}
\end{equation*}
If $\alpha=1$ and $\beta=\pi/2$, then $K=\pi/2$ and we still have \eqref{equ-cst-rcx}. Also from \eqref{equ-area-bound} we have
\begin{equation*}
\Area(S\cap A(\cos z)^c)<361.
\end{equation*}
This finishes the proof of Theorem \ref{thm-sin}.
\hfill $\square$

%----------------------------------------------------------------------------------------------------------------

\end{document}